\documentclass[12pt]{amsart}
\usepackage{amssymb}
\usepackage{amsfonts,amscd}
\usepackage{amsmath}
\usepackage[backref=page]{hyperref}
 \numberwithin {equation}{section}
\newtheorem{theo}{Theorem}[section]

\newtheorem{corol}[theo]{Corollary}
\newtheorem{prop}[theo]{Proposition}
\theoremstyle{definition}
\newtheorem{remark}[theo]{Remark}
\newtheorem{remarks}[theo]{Remarks}

\newtheorem{defi}[theo]{Definition}
\newcommand{\bequ}{\begin{equation}} %
\newcommand{\eequ}{\end{equation}}
\newcommand{\bequs}{\begin{equation*}}
\newcommand{\eequs}{\end{equation*}}
\newcommand{\beqs}{\begin{equation*}}
\newcommand{\eeqs}{\end{equation*}}
\newcommand{\Real}{\mathbb R}
\newcommand{\Nat}{\mathbb N}
\newcommand{\vphi}{\varphi}
\newcommand{\epsic}{\varepsilon}

\newcommand{\lbda}{\lambda}
\newcommand{\Lra}{\ \Longrightarrow\ }
\newcommand{\rar}{\rightarrow}

\newcommand{\lrar}{\leftrightarrow}
\newcommand{\Ra}{\;\Rightarrow\;}
\newcommand{\sms}{\smallsetminus}

\newcommand{\sse}{\subseteq}
\newcommand{\mbx}{\mbox}
\newcommand{\ov}{\overline}
\newcommand{\un}{\underline}
\newcommand{\ety}{\emptyset}
\newcommand{\dom}{\operatorname{dom}}
\newcommand{\ben}{\begin{enumerate}}
\newcommand{\een}{\end{enumerate}}
\newcommand{\eite}{\end{itemize}}
\newcommand{\bite}{\begin{itemize}}
\newcommand{\q}{\quad}

\newcommand{\rond}{\mathcal}
\newcommand{\rd}{\mathcal}
\begin{document}
\title{Ekeland, Takahashi and Caristi principles in  quasi-pseudometric spaces}
\author{S. Cobza\c{s} }
\address{\it Babe\c s-Bolyai University, Faculty of Mathematics
and Computer Science, 400 084 Cluj-Napoca, Romania}\;\;
\email{scobzas@math.ubbcluj.ro}
\begin{abstract}  We prove  versions  of Ekeland, Takahashi and Caristi principles in sequentially right $K$-complete quasi-pseudometric spaces (meaning asymmetric pseudometric spaces),  the equivalence  between these principles, as well as their equivalence to  the completeness of the underlying quasi-pseudometric space.

The key tools    are Picard sequences for  some special set-valued mappings corresponding to a function $\vphi$ on a quasi-pseudometric space,    allowing  a unitary treatment of  all these principles.\medskip

\textbf{Classification MSC 2010:} 58E30 47H10  54E25 54E35 54E50 \medskip


\end{abstract}

\date{\today}
\maketitle
\section{Introduction}

Ivar Ekeland announced  in 1972, \cite{ekl72}  (the proof appeared  in 1974, \cite{ekl74}) a theorem asserting the existence of the minimum of a small perturbation of a lower semicontinuous (lsc) function defined on a complete metric space. This result, known as Ekeland Variational Principle (EkVP), proved to be a very versatile tool in various areas of mathematics and applications  - optimization theory, geometry of Banach spaces, optimal control theory, economics, social sciences, and others. Some of these applications are presented by Ekeland himself in \cite{ekl79}.

At the same time, it turned out that this principle is equivalent to a lot  of results in fixed point theory (Caristi fixed point theorem), geometry of Banach spaces  (drop property), and  others (see  \cite{penot86}, for instance).
Takahashi \cite{taka91} (see also \cite{Taka00}) found a sufficient condition for the existence of the minimum of a lsc function on a complete metric space, known as Takahashi minimization principle, which also turned  to be equivalent to  EkVP (see \cite{taka91} and \cite{hamel94}).

For convenience, we mention these three principles.

\begin{theo}[Ekeland, Takahashi and Caristi principles]\label{t1.Ek-Taka-Car} Let $(X,d)$ be a  complete  metric space and $\vphi:X\to\Real\cup\{\infty\}$ a proper bounded below  lsc function.
Then the following hold.
\ben
\item[\rm (wEk)] There exists $z\in X$ such that $\vphi(z)<\vphi(x)+d(x,z)$ for all $x\in X\sms \{z\}.$
\item[\rm (Tak)] If for every $x\in X$  with $\vphi(x)>\inf\vphi(X)$   there exists  an element $y\in X\sms\{x\}$ such that $\vphi(y)+d(x,y)\le\vphi(x), $ then $\vphi$ attains its minimum on $X$, i.e., there  exists
$z\in X$ such that $\vphi(z)=\inf\vphi(X).$
\item[\rm (Car)] If the mapping  $T:X\to X$ satisfies $d(Tx,x)+\vphi(Tx)\le\vphi(x) $ for all $x\in X,$ then $T$ has a fixed point in $X$, i.e.,   there  exists
$z\in X$ such that $Tz=z.$
\een\end{theo}

The statement (wEk) is called the weak form of the Ekeland variational principle.  Later, various versions and extensions of these principles appeared, a good record (up to 2009) being given in the book \cite{Meghea}.

Some versions of EkVP and Takahashi minimum principles in $T_1$ quasi-metric spaces were proved in \cite{cobz11} and \cite{kassay19}, respectively. In \cite{cobz11}  the equivalence of  the weak Ekeland variational principle to Caristi fixed point theorem was proved  and the implications of the validity of  Caristi fixed point theorem principle on the completeness of the underlying quasi-pseudometric space were studied as well.  In \cite{kassay19}  the same is done for the weak form of Ekeland variational principle and   Takahashi minimization principle.
The extension of wEk to arbitrary quasi-metric spaces was given in \cite{karap-romag15}.

Other extensions,  with applications to various areas of social sciences and psychology, were
given in \cite{bcs18}--\cite{bao-mord-soub15c}, \cite{soub16},  \cite{soub-qiu18}.

  All these extensions
are obtained by  relaxing the conditions, else on the minimized function or   on the underlying metric space, or both. In this paper we  consider both  approaches -- we study the validity of EkVP in quasi-pseudometric spaces for functions satisfying a weaker semicontinuity  condition, called near lower semicontinuity  (see Subsection \ref{Ss.lsc}). Roughly speaking, a quasi-pseudometric   is a function $d$ on $X\times X$, where $X$ is a set, satisfying all the axioms of a pseudometric but symmetry, that is, the possibility that $d(x,y)\ne d(y,x)$ for some $x,y\in X$ is not excluded (see Section \ref{S.qm}). We also prove (Theorem \ref{t.Ek2-qm}) a full version of EkVP in quasi-pseudometric spaces for lsc functions.

The nearly lsc  functions were introduced in \cite{karap-romag15}. We  show that a nearly lsc function is  lsc if and only if it is monotone   with respect to the specialization order (see Subsection \ref{Ss.top-ord} and Proposition \ref{p1.lsc-n}).

The main results of the paper are contained in Section \ref{S.T0-qm}, where one proves quasi-pseudometric versions of Ekeland, Takahashi and Caristi principles  and
their equivalence. One proves also that the validity of wEkVP implies the sequential right-$K$-completeness of the underlying quasi-pseudometric space. We conclude this section by showing that the $T_1$ versions of these principles   are particular cases of those   proved for general quasi-pseudometric spaces.

The key tools used in the  proofs of these results are Picard sequences for  some special set-valued mappings corresponding to a function $\vphi$ on a quasi-pseudometric space (see Subsection \ref{Ss.Pic-seq}), which allow  a unitary treatment of  all these principles. The idea to use Picard sequences appeared in \cite{dhm83} and was subsequently exploited in \cite{bt15} and \cite{bcs18}.

\section{Quasi-metric spaces}\label{S.qm}
\subsection{Topological properties}

 A {\it quasi-pseudometric} on an arbitrary set $X$ is a mapping $d: X\times X\to
[0,\infty)$ satisfying the following conditions:
\begin{align*}%
\mbox{(QM1)}&\qquad d(x,y)\geq 0, \quad and  \quad d(x,x)=0;\\
\mbox{(QM2)}&\qquad d(x,z)\leq d(x,y)+d(y,z),  %
\end{align*}%
for all $x,y,z\in X.$ If further
$$%
\mbox{(QM3)}\qquad d(x,y)=d(y,x)=0\Rightarrow x=y,
$$%
for all $x,y\in X,$ then $d$ is called a {\it quasi-metric}. The pair $(X,d)$ is called a {\it
quasi-pseudometric space}, respectively  a {\it quasi-metric space}\footnote{In \cite{Cobzas} the term ``quasi-semimetric" is used instead of ``quasi-pseudometric"}. The conjugate of the quasi-pseudometric
$d$ is the quasi-pseudometric $\bar d(x,y)=d(y,x),\, x,y\in X.$ The mapping $
d^s(x,y)=\max\{d(x,y),\bar d(x,y)\},\,$ $ x,y\in X,$ is a pseudometric on $X$ which is a metric if and
only if $d$ is a quasi-metric.

If $(X,d)$ is a quasi-pseudometric space, then for $x\in X$ and $r>0$ we define the balls in $X$ by the formulae %
\begin{align*}%
B_d(x,r)=&\{y\in X : d(x,y)<r\} \; \mbox{-\; the open ball, and }\\ %
B_d[x,r]=&\{y\in X : d(x,y)\leq r\} \; \mbox{-\; the closed ball. } %
\end{align*} %

 The topology $\tau_d$ (or $\tau(d)$) of a quasi-pseudometric space $(X,d)$ can be defined starting from the family
$\rond{V}_d(x)$ of neighborhoods  of an arbitrary  point $x\in X$:%
\bequs
\begin{aligned}
V\in \rond{V}_d(x)\;&\iff \; \exists r>0\;\mbox{such that}\; B_d(x,r)\subseteq V\\
                             &\iff \; \exists r'>0\;\mbox{such that}\; B_d[x,r']\subseteq V. %
\end{aligned} %
\eequs

The convergence of a sequence $(x_n)$ to $x$ with respect to $\tau_d,$ called $d$-convergence and
denoted by
$x_n\xrightarrow{d}x,$ can be characterized in the following way %
\bequ\label{char-rho-conv1} %
         x_n\xrightarrow{d}x\;\iff\; d(x,x_n)\to 0. %
\eequ %

Also
\bequ\label{char-rho-conv2} %
         x_n\xrightarrow{\bar d}x\;\iff\;\bar d(x,x_n)\to 0\; \iff\; d(x_n,x)\to 0. %
\eequ %

As a space equipped with two topologies, $\tau_d$ and   $\tau_{\bar d}\,$, a quasi-pseudometric space can be viewed as a bitopological space in the sense of Kelly \cite{kelly63}.

The following   topological properties are true for  quasi-pseudometric spaces.
    \begin{prop}[see \cite{Cobzas}]\label{p.top-qsm1}
   If $(X,d)$ is a quasi-pseudometric space, then the following hold.
   \begin{enumerate}
   \item[\rm 1.] The ball $B_d(x,r)$ is $\tau_d$-open and  the ball $B_d[x,r]$ is
       $\tau_{\bar{d}}$-closed. The ball    $B_d[x,r]$ need not be $\tau_d$-closed.
     \item[\rm 2.]
   The topology $\tau_d$ is $T_0$ if and only if  $d $ is a quasi-metric.   \\ The topology $\tau_d$ is $T_1$ if and only if
   $d(x,y)>0$ for all  $x\neq y$  in $X$.
      \item [\rm 3.]  For every fixed $x\in X,$ the mapping $d(x,\cdot):X\to (\Real,|\cdot|)$ is
   $\tau_d$-usc and $\tau_{\bar d}$-lsc. \\
   For every fixed $y\in X,$ the mapping $d(\cdot,y):X\to (\Real,|\cdot|)$ is $\tau_d$-lsc and
   $\tau_{\bar d}$-usc.

  \end{enumerate}%
     \end{prop} %

The following remarks show that imposing too many conditions on a quasi-pseudometric space it becomes  pseudometrizable.

\begin{remark}[\cite{kelly63}] Let $(X,d)$ be  a quasi-metric space. Then
\ben
  \item[\rm (a)]\  if  the mapping $d(x,\cdot):X\to (\Real,|\cdot|)$ is $\tau_d$-continuous for   every $x\in X,$ then the topology $\tau_d$ is regular;
   \item[\rm (b)] if $\tau_d\subseteq \tau_{\bar{d}}$, then the topology $\tau_{\bar{d}}$ is pseudometrizable;
 \item[\rm (c)] if $d(x,\cdot):X\to (\Real,|\cdot|)$ is $\tau_{\bar{d}}$-continuous for every $x\in X,$ then the topology $\tau_{\bar{d}}$ is pseudometrizable.
\een\end{remark}

\begin{remark}
  The characterization of Hausdorff property (or $T_2$) of quasi-metric spaces can be given in terms of uniqueness of the limits, as in the metric case. The topology of a quasi-pseudometric space $(X,d)$ is Hasudorff if and only if every sequence in $X$ has at most one $d$-limit if and only if every sequence in $X$ has at most one $\bar d$-limit (see \cite{wilson31}).

   In the case of an asymmetric normed space   there exists a characterization in terms of the quasi-norm (see \cite{Cobzas}, Propositions 1.1.40).
  \end{remark}

  Recall that a topological space $(X,\tau)$  is called:

 \begin{itemize}
 \item  $T_0$ if for every pair of distinct points  in $X$, at least one of them has a neighborhood   not containing the other;
   \item  $T_1$ if   every pair of distinct points in $X$,  each of them has a neighborhood   not containing the other;
\item    $T_2$ (or {\it Hausdorff}) if  every  two distinct points  in $X$ admit  disjoint  neighborhoods;
\item  {\it regular}  if for every   point $x\in X$ and closed set $A$ not containing $x$ there exist the disjoint open sets $U,V$ such that $x\in U$  and $A\subseteq V.$  \end{itemize}

  \subsection{Completeness in quasi-metric spaces}

   The lack of symmetry in the definition of quasi-metric spaces causes a lot of troubles, mainly concerning
   completeness, compactness and total boundedness in such spaces.
   There are a lot of completeness notions in quasi-metric spaces, all agreeing with the usual notion of
   completeness in the metric case,  each of
   them having its advantages and weaknesses (see \cite{reily-subram82}, or \cite{Cobzas}).

   As in what follows we shall shall work only with two of these  notions, we shall present only them,
   referring to \cite{Cobzas} for   others.

   We use the  notation
   \begin{align*}
  &\Nat=\{1,2,\dots\}  \mbx{ --  the set of natural numbers,}\\
&\Nat_0=\Nat\cup\{0\} \mbx{ --  the set of non-negative integers.}
  \end{align*}

  \begin{defi}\label{def.rKC}
    Let $(X,d)$ be a quasi-pseudometric space. A sequence $(x_n)$ in $(X,d)$ is called:
\bite
\item\;  {\it left $d$-$K$-Cauchy} if  for every $\epsic >0$
   there exists $n_\epsic\in \Nat$ such that
\bequ\label{def.l-Cauchy}\begin{aligned} %
 &\forall n,m, \;\;\mbox{with}\;\; n_\epsic\leq n < m ,\quad d(x_n,x_m)<\epsic\\ %
 \iff &\forall n \geq n_\epsic,\; \forall k\in\Nat,\quad d(x_n,x_{n+k})<\epsic; %
 \end{aligned}\eequ
\item\;   {\it right $d$-$K$-Cauchy} if  for every $\epsic >0$
   there exists $n_\epsic\in \Nat$ such that
\bequ\label{def.r-Cauchy}\begin{aligned} %
 &\forall n,m, \;\;\mbox{with}\;\; n_\epsic\leq n < m ,\quad d(x_m,x_n)<\epsic\\ %
 \iff &\forall n \geq n_\epsic,\; \forall k\in\Nat,\quad d(x_{n+k},x_n)<\epsic. %
 \end{aligned}\eequ\eite
  \end{defi}

The quasi-pseudometric space $(X,d)$ is called:\bite\item\;  {\it sequentially left $d$-$K$-complete} if every left $d$-$K$-Cauchy is
$d$-convergent;
\item\;  {\it sequentially  right $d$-$K$-complete} if every right  $d$-$K$-Cauchy is
$d$-convergent.\eite

\begin{remarks}\hfill\begin{itemize}
\item[{\rm 1.}]
It is obvious that a sequence is left $d$-$K$-Cauchy
   if and only if it is right $\bar{d}$-$K$-Cauchy.
\item[{\rm 2.}] There are examples showing that a $d$-convergent sequence need not  be
left $d$-$K$-Cauchy, showing that in the asymmetric case the situation is far more complicated than in the
symmetric one (see \cite{reily-subram82}).
\item[{\rm 3.}]
  If each convergent sequence in a regular quasi-metric space $(X,d)$
   admits a left $K$-Cauchy subsequence, then $X$ is metrizable
   (\cite{kunzi-reily93}).
   \end{itemize} \end{remarks}

\begin{prop}\label{p1.rKC} Let $(X,d)$ be a quasi-pseudometric space. If a right $K$-Cauchy sequence $ (x_n)$ contains a subsequence convergent to some $x\in X$, then the sequence $(x_n)$ converges to $x$.
\end{prop}
\begin{remark}
  One can define  more general notions of completeness by replacing in Definition \ref{def.rKC} the sequences with nets. Stoltenberg \cite[Example 2.4]{stolt69} gave an example of a sequentially right $K$-complete $T_1$ quasi-metric space which is not right $K$-complete (i.e., not right $K$-complete by nets).
\end{remark}

\textbf{Convention.} \emph{In the following, when speaking about metric or topological properties in a quasi-pseudometric space $(X,d)$ we shall always  understand those corresponding to $d$ and we shall omit $d$ or $\tau_d$, i.e., we shall write ``$(x_n)$ is right $K$-Cauchy" instead of ``$(x_n)$ is right $d$-$K$-Cauchy",
$\ov A$ instead of $\ov A^d$, etc.}

\subsection{The specialization order in  topological spaces}\label{Ss.top-ord}
Let $(X,\tau)$ be  a topological space. Denote by $\rd{V}(x)$ the family of all neighborhoods of a point $x\in X$. The \emph{specialization order}  in $X$ is the partial order defined by
\begin{equation}\label{def.top-ord}\begin{aligned}
  x\le_\tau y &\iff x\in\overline{\{y\}}\\
  &\iff \forall V\in\rd{V}(x),\; y\in V,
\end{aligned}\end{equation}
that is $y$ belongs to every open set containing $x$.

By a \emph{preorder} on a set $X$ we understand a relation $\le$ on $ X$ such that

(O1)\quad$x\le x$\quad and

(O2)\quad$((x\le y)\wedge (y\le z))\Ra x\le z,$\\
for all $x,y,z\in X.$
If further

(O3)\quad$((x\le y)\wedge (y\le x))\Ra x=y,$\\
then $\le$ is called an \emph{order} on $X$.

\begin{prop}\label{p1.top-ord}
  Let $(X,\tau)$ be a topological space.  Then
  \ben
  \item[\rm (i)] the relation defined by \eqref{def.top-ord} is a preorder on $X$;
   \item[\rm (ii)] it is an order if and only if the topology $\tau$ is $T_0$;
 \item[\rm (iii)] the topology  $\tau$ is $T_1$ if and only if  $\,\le_\tau$ is the equality relation in $X$.
\een\end{prop}\begin{proof} (i) \; Since $x\in\overline{\{x\}}$ it follows $x\le_\tau x$.

The transitivity follows from the following implication
$$
x\in  \overline{\{y\}}\;\mbox{ and }\; \{y\}\subseteq\overline{\{z\}}\;\Rightarrow\; x\in\overline{\{y\}}\subseteq\overline{\overline{\{z\}}}=\overline{\{z\}}\,,
$$
that is
$$
x\le_\tau y \;\mbox{ and }\;y\le_\tau z \;\Rightarrow\; x\le_\tau z\,.$$

(ii)\; The antisymmetry means that
$$
x\le_\tau y \;\mbox{ and }\;y\le_\tau x \;\Rightarrow\; x=y\,,$$
 or, equivalently,
$$
 x\ne y  \;\Rightarrow\; x\nleq_\tau y \;\mbox{ or }\;y\nleq_\tau x \,,$$
 for all $x,y\in X.$

 But
 \begin{align*}
 x\nleq_\tau y \;\mbox{ or }\;y\nleq_\tau x &\iff x\notin\ov{\{y\}} \;\mbox{ or }\;y\notin\ov{\{x\}}\\
 &\iff \exists V\in\rd{V}(x),\; y\notin V \;\mbox{ or }\; \exists U\in\rd{V}(y),\; x\notin U\\
  &\iff \tau\;\mbx{ is }\; T_0\,.
 \end{align*}

(iii)\; The topological space  $X$ is $T_1$ if and only if  $\,\overline{\{x\}}=\{x\}$ for every $x\in X$. Consequently,
$$
x\le_\tau y\iff x\in \overline{\{y\}}=\{y\}\iff x=y\,,
$$

Conversely,
$$
x\le_\tau y\iff x=y\,,
$$
is equivalent to
\begin{align*}
 x\in \overline{\{y\}}&\iff x=y\,,
 \end{align*}
hence $\overline{\{y\}}=\{y\}$ for all $y\in X$, that is, $\tau$ is $T_1.$
 \end{proof}

  Let $(X,\le)$ be an ordered set. For $A\sse X$ put
 \begin{align*}
   \uparrow\!\! A&=\{y\in X :\exists x\in A,\, x\le y\}\quad\mbx{and}\\
   \downarrow\!\! A&=\{y\in X :\exists x\in A,\, y\le x\}
 \end{align*}

In the following results the  order notions are considered with respect to  the specialization order $\le_\tau$.
\begin{prop}\label{p2.top-ord}
  Let $(X,\tau)$ be a topological space and $A\subseteq X$.
  \begin{enumerate}
  \item[\rm 1.] If the set $A$ is  open, then it is upward closed, i.e. $\uparrow\!\! A=A$.
  \item[\rm 2.] If the set $A$ is  closed, then it is downward closed, i.e. $\downarrow\!\! A=A$.
  \end{enumerate}
\end{prop}  \begin{proof}
  1.\; It is a direct consequence of definitions. Let  $x\in A$ and $y\in X,\, x\le_\tau y.$ Since $A$ is open, this inequality   implies $y\in A.$

  2. \;  Let  $x\in A$ and $y\in X,\, y\le_\tau x.$ Then $y\in\overline{\{x\}}\subseteq \overline A=A$.
\end{proof}

Let us define the \emph{saturation} of a subset $A$ of $X$ as the intersection of all open subsets of $X$ containing $A$. The set $A$ is called \emph{saturated} if  equals its saturation.
\begin{prop}\label{p3.top-ord}
  Let $(X,\tau)$ be a topological space.
  \begin{enumerate}
  \item[\rm 1.] For every $x\in X$, $\downarrow\!\! x=\overline{\{x\}}$.
  \item[\rm 2.]  For any subset $A$ of $X$ the saturation of $A$ coincides with $\uparrow\!\! A$.
  \end{enumerate}
\end{prop}\begin{proof}
 1.\; This follows from the equivalence
 $$
 y\le_\tau x\iff y\in\overline{\{x\}}\,.$$

 2.\; Since every open set is upward closed, $U\in \tau$ and $U\supset A$ implies $U\supset\uparrow\!\! A,$
 that is
 $$
 \uparrow\!\! A\subseteq \bigcap\left\{U\in\tau : A\subseteq U\right\}\,.
 $$

 If $y\notin \uparrow\!\! A$, then for every $x\in A$ there exists $U_x\in \tau$ such that  $x\in U_x$ and  $y\notin U_x.$ It follows $y\notin V:=\bigcup\{U_x : x\in A\}\in\tau$ and $ A\subseteq V$, hence $y\notin \bigcap\{U \in \tau : A\subseteq U\}$, showing that
$$
\complement\left(\uparrow\!\!A\right)\subseteq \complement\left(\bigcap\left\{U\in\tau : A\subseteq U\right\}\right)\iff \bigcap\left\{U\in\tau : A\subseteq U\right\}\subseteq  \uparrow\!\!A\,.$$
\end{proof}

\subsection{Lower semi-continuous functions on quasi-pseudometric spaces}\label{Ss.lsc}
Let  $(X,d)$ be a quasi-pseudometric space.
The specialization order $\le_d$ in $X$ with respect  to the topology $\tau_d$  takes the form
\bequ\label{def.ord-qm}
x\le_dy\iff d(x,y)=0\,,
\eequ
for $x,y\in X.$

Indeed, $$\begin{aligned}
 x\le_d y&\iff x\in\ov{\{y\}}\\ &\iff \forall r>0,\, y\in B_d(x,r)\\&\iff \forall r>0,\, d(x,y)<r\\&\iff d(x,y)=0\,.
\end{aligned}$$

For reader's convenience,  we present some remarks about $\liminf$ and $\limsup$ of sequences in $\Real.$
 Let $(a_n)$ be a sequence in $\Real.$ For $n\in\Nat$ let
 $$\ov{a}_n=\sup\{a_k:\ k\ge n\}
\quad\mbox{and}\quad
\un{a}_n=\inf\{a_k:\ k\ge n\}.$$

It follows
$\ov{a}_{n+1}\le \ov{a}_n$ and $\un{a}_{n+1}\ge \un{a}_n$.

By definition one puts
\begin{align*}
  \limsup_{n\to\infty} a_n&=\lim_{n\to\infty }\,\ov{a}_n=\inf\{\ov{a}_n : n\in\Nat\}
\quad\mbox{and}\\ \liminf_{n\to\infty} a_n&=\lim\limits_{n\to\infty}\,\un{a}_n=\sup\{\un{a}_n : n\in\Nat\}.
\end{align*}

  Note that $\ov{l}=\liminf_n a_n$ and $\un{l}=\liminf_n a_n$
always exist, $\un{l}\le \ov{l}$ and   the sequence $(a_n)$ has
the limit $l$ if and only if
$$\liminf_{n\to\infty}\,a_n=\limsup_{n\to\infty}\,a_n=l.$$

A {\it cluster point} of the sequence $(a_n)$ is a number $x\in
\ov{\mathbb{R}}$ such that $\lim\limits_{k\to\infty }\,a_{n_k}=x$,
for some subsequence $(a_{n_k})$ of $(a_n)$. The numbers $\ov{l}$
and $\un{l}$ are cluster points of the sequence $(a_n)$ and any
other cluster point $\lbda$ of $(a_n)$ satisfies the inequalities
$$\un{l}\le \lbda\le \ov{l}.$$

 A function $f:X\to\Real\cup\{\infty\}$ is called:
 \bite
\item \emph{lower semi-continuous} (lsc) at $x\in X$ if for every sequence $(x_n)$ in $X$ converging to $x$,
\bequ\label{def.lsc}
f(x)\le\liminf_{n\to\infty}f(x_n);
\eequ
 \item \emph{nearly lower semi-continuous} (nearly lsc) at $x\in X$ if \eqref{def.lsc} holds only for  sequences $(x_n)$ with distinct terms converging to $x$;
 \item \emph{lower semi-continuous} (\emph{nearly lower semi-continuous}) on $X$ if it is lsc (nearly lsc) at every $x\in X.$
 \eite

 Obviously, a lsc function is nearly lsc. The notion of nearly lsc function was introduced by Karapinar and Romaguera \cite{karap-romag15} who showed by an example that it is effectively  more general than lsc.
 We call a function  $f:X\to\Real\cup\{\infty\}$ $d$-monotone if
 \bequs
 x\le _d y \Ra f(x)\le f(y)\,,\eequs
 for all $x,y\in X.$
 \begin{prop}\label{p1.lsc-n} Let  $(X,d)$ be a quasi-pseudometric space and $f:X\to\Real\cup\{\infty\}$ a function.
 \ben
 \item[\rm 1.] The function $f$ is nearly lsc at $x\in X$ if and only if \eqref{def.lsc} holds for all sequences $(x_n)$  without constant subsequences such that $\lim_{n\to\infty} x_n=x.$
 \item[\rm 2.] The function $f$ is lsc   if and only if it is nearly lsc and $d$-monotone.
 \item[\rm 3.]  If the topology $\tau_d$ is $T_1$, then any nearly lsc function is lsc.
 \een\end{prop}\begin{proof}
   1. Suppose that  \eqref{def.lsc} holds for all sequences with distinct terms converging to $x$. Let $(x_n)$ be a sequence without constant subsequences converging to
   $x$. Put $n_1=1$ and define inductively
   $$
   n_{k+1}=\min\{n:x_n\notin\{x_{n_1},\dots,x_{n_k}\}\},\; \; k\in\Nat.$$

   Since every term $x_i$ of the sequence $(x_n)$ appears only finitely many times, the numbers $n_k$ are well defined,
   \begin{align*}
     &n_1<n_2<\dots\\
     &x_{n_i}\ne x_{n_j}\;\mbx{ for }\; i\ne j\;\mbx{ and}\\
     &\{x_n : n\in\Nat\}=\{x_{n_k} :k\in \Nat\}\,.
   \end{align*}
   Let also
   $$
   m_{k}=\max\{i : x_i\in \{x_{n_1},\dots,x_{n_k}\}\},\quad k\in\Nat.$$

   Then $$m_k\le m_{k+1}\quad\mbx{and}\quad m_k\ge n_k\,,$$
   for all $k\in\Nat.$ It follows
   $$
   \{x_j : j>m_k\}\sse \{x_{n_i} : i>k\}\,,$$
   so that
   $$
   \inf\{f(x_j) : j>m_k\}\ge \inf\{f(x_{n_i}) : i>k\}\,,$$
   for all $k\in\Nat.$ Since  $\lim_{k\to\infty}m_k=\infty$,
   \begin{align*}
     \liminf_{k\to\infty}f(x_{n_k})&=\lim_{k\to\infty}\, \inf_{i>k} f(x_{n_i})\\
     &\le \lim_{k\to\infty}\, \inf_{j>m_k}f(x_j)\\
     &=\liminf_{n\to\infty}f(x_n)\le \liminf_{k\to\infty}f(x_{n_k})\,.
   \end{align*}

   Consequently, $$
    f(x)\le \liminf_{k\to\infty}f(x_{n_k})=\liminf_{n\to\infty}f(x_n).$$

   2. We have remarked that  a lsc function is nearly lsc. We show that it  is also $d$-monotone. Indeed, if $x\le_d y$, that is $d(x,y)=0$, then the constant sequence $x_n=y,\, n\in\Nat,$ converges to $x$, so that, by the lsc of the function $f$,
   $$
   f(x)\le\lim_{n\to\infty}f(x_n)=f(y)\,.$$

   Suppose now that $f:X\to\Real\cup\{\infty\}$ is $d$-monotone and nearly lsc. The proof of the fact that it is lsc will be based on the following remark.
  \medskip

  \emph{ Claim }I. \emph{If} $f:X\to\Real\cup\{\infty\}$ \emph{is} $d$-\emph{monotone and}
  \bequ\label{eq0.n-lsc}
  x_n\to x\quad\mbx{{\it and}}\quad  f(x_n)\to \lbda\; \Lra\; f(x)\le\lbda\,,
  \eequ
  \emph{for every sequence $(x_n)$ without constant subsequences,  then \eqref{eq0.n-lsc} holds for arbitrary sequences in} $X$.
  \medskip

  Let $(x_n)_{n\in\Nat}$  be an arbitrary sequence in $X$ such that $x_n\to x$ and $f(x_n)\to\lbda$  as $n\to\infty.$
     If $(x_n)$ contains a constant subsequence, say $x_{n_k}=y,\, k\in \Nat,$  then   $\lim_{k\to\infty}f(x_{n_k})=\lbda$ implies $f(y)=f(x_{n_k})=\lbda$ for sufficiently large $k\in\Nat.$

   Also
     $$
   d(x,y)=\lim_{k\to\infty}d(x,x_{n_k}) =0\,,
   $$
  implies $x\le_d y$, so that, by the $d$-monotony of $f$,
   $$
   f(x)\le f(y)=\lbda =\lim_{n\to\infty}f(x_{n})\,.$$

   Let now $(x_n)$ be an arbitrary sequence in $X$ converging to $x$ and $\lbda=\liminf_{n\to\infty}f(x_n).$ Then there exists a subsequence $(x_{n_k})_{k\in\Nat} $
  of $(x_n)$ such that $\lim_{k\to\infty}f(x_{n_k})=\lbda.$ By Claim I,

  $$
   f(x)\le \lbda =\liminf_{n\to\infty}f(x_{n})\,,$$
   showing that $f$ is lsc.

 3. Suppose that $f:X\to\ov\Real$ is nearly lsc. If the topology $\tau(d)$ is $T_1$, then, by Proposition \ref{p1.top-ord}.(iii), the specialization order is the equality on $X$, hence  the
 function $f $  is $d$-monotone. By  2  this implies that $f$ is lsc.
   \end{proof}

 \subsection{Picard sequences in quasi-pseudometric spaces}\label{Ss.Pic-seq}

 Let $(X,d)$ be a quasi-pseudometric space and $\vphi:X\to\Real\cup\{\infty\}$ a function. For $x\in \dom \vphi:=\{x\in X : \vphi(x)<\infty\}$ define the set $S(x)$ by
\bequs
S(x)=\{y\in X : \vphi(y)+d(y,x)\le \vphi(x)\}\,.\eequs

The function $\vphi$ is called \emph{proper} if $\dom\vphi\ne\ety.$

\begin{prop}\label{p1.Sx}
 The sets $S(x)$ have the following properties:
\bequ\label{eq2.Sx}
\begin{aligned}
  {\rm (i)}\;\; &x\in S(x)\quad\mbx{and}\quad S(x)\sse\dom \vphi;\\
  {\rm (ii)}\;\; &y\in S(x) \Ra \vphi(y)\le \vphi(x)\;\mbx{ and }\; S(y)\subseteq S(x);\\
  {\rm (iii)}\;\; &y\in S(x)\sms\ov{\{x\}} \Ra \vphi(y)<\vphi(x);\\
  {\rm (iv)}\;\; &\mbx{if $\vphi$ is bounded below, then}\\&S(x)\sms \ov{\{x\}}\ne\ety \Ra \vphi(x)>\inf \vphi(S(x));\\
  {\rm (v)}\;\; &\mbx{if $\vphi$ is lsc, then }  S(x)\mbx{ is closed}.
  \end{aligned}\eequ
\end{prop}\begin{proof}
    The relations (i) are immediate consequences of  the definition of $S(x)$.

(ii)\; If $y\in S(X)$, then $0\le d(y,x)\le \vphi(x)-\vphi(y)$ implies $\vphi(y)\le \vphi(x).$  Let now $y\in S(x)$ and $z\in S(y)$. Then
  \begin{align*}
  \vphi(z)+d(z,x)&\le [\vphi(z)+ d(z,y)] +d(y,x) \\
 &\le \vphi(y)+d(y,x)\le \vphi(x)\,\qquad (\vphi(z)+ d(z,y)\le \vphi(y)\; \mbx{as } z\in S(y)),
  \end{align*}
showing that  $z\in S(x).$

 (iii) Follows from the inequalities:
$$
0<d(y,x)\le \vphi(x)-\vphi(y)\,.$$

(iv) If there exists $y\in  S(x)\sms\ov{\{x\}}$, then, by (iii),
$$
\inf \vphi(S(x))\le \vphi(y)< \vphi(x)\,.$$

 (v) Follows from the lsc of the function $\vphi$ and the $d$-lsc  of $d(\cdot,x)$.
\end{proof}

 \begin{remark}\hfill\ben
\item[\rm 1.] If $\vphi(x)=\infty$ then $S(x) = X$, so it is natural to consider $S(x)$ only for points $x$ in the domain of $\vphi$.

\item[\rm 2.]   Considering the   order $\preceq$ on $X$ given by
  $$
  x\preceq y \iff \vphi(y)+d(y,x)\le \vphi(x)\,,$$
  we have
  $$
  S(x)=\{y\in X : x\preceq y\}\,.$$

\item[\rm 3.] It is worth to mention that sets of this kind  were used by Penot \cite{penot77} as early as 1977 in a proof of Caristi fixed point theorem in complete metric spaces.
  \een\end{remark}

 A \emph{Picard sequence} corresponding to  a set-valued mapping  $F:X\rightrightarrows X $ is a sequence $(x_n)_{n\in\Nat_0}$ such that
  \bequs
  x_{n+1}\in F(x_n)\;\mbx{ for all }\; n\in\Nat_0\,,
  \eequs
 for a given initial point $x_0\in X.$  This notion was introduced in \cite{dhm83} (see also \cite{bt15}).

  \begin{prop}[Picard sequences]\label{p1.Pic} Let $(X,d)$ be a quasi-pseudometric space and $\vphi:X\to\Real\cup\{\infty\}$ a proper bounded below function.
  For $x\in\dom\vphi$ let
  $$
  S(x)=\{y\in X : \vphi(y)+d(y,x)\le\vphi(x)\}\quad\mbx{and}\q J(x)=\inf\vphi(X)\,.$$
  Let $x_0\in\dom\vphi.$ We distinguish two situations.
  \ben
  \item[\rm 1.] There exists $m\in \Nat_0$ such that
   \bequ\label{eq01.Pic}\begin{aligned}
   {\rm (i)}\quad&\vphi(x_k)>J(x_k); \\
   {\rm (ii)}\quad& x_{k+1}\in S(x_k) \;\mbx{ and }\; \vphi(x_{k+1})<(\vphi(x_k)+J(x_k))/2\\
   {\rm (iii)}\quad& \vphi(x_m)=J(x_m)\,.
   \end{aligned}\eequ
   for all $\,0\le k\le m-1.$

Putting   $z=x_m$,  the following conditions are satisfied:
  \bequ\label{eq02.Pic}\begin{aligned}
   {\rm (i)}\quad& S(x_{k+1})\sse S(x_k)\;\mbx{ and }\; \vphi(x_{k+1}<\vphi(x_k)  \;\mbx{ for all }\; 0\le k\le m-1;\\
   {\rm (ii)}\quad& z\in S(x_k)\;\mbx{and }\; S(z)\sse S(x_k)\;\mbx{ for }\; 0\le k\le m;\\
    {\rm (iii)}\quad& \vphi(y)=\vphi(z)=J(z)\;\mbx{ and} \\
    {\rm (iv)}\quad&    S(y)\sse\ov{\{y\}}\:\mbx{ for all }\; y\in S(z)\,.
  \end{aligned}\eequ

  \item[\rm 2.]  There exists a sequence $(x_n)_{n\in\Nat_0}$ such that
   \bequ\label{eq1.Pic}\begin{aligned}
   {\rm (i)}\quad&\vphi(x_n)>J(x_n);\\
   {\rm (ii)}\quad& x_{n+1}\in S(x_n) \;\mbx{ and }\; \vphi(x_{n+1})<(\vphi(x_n)+J(x_n))/2\,,\\
    \end{aligned}\eequ
   for all $ n\in\Nat_0\,.$

  Then  the sequence $(x_n)_{n\in\Nat_0}$ satisfies the conditions
  \bequ\label{eq2.Pic}  \begin{aligned}
    {\rm (i)}\quad &S(x_{n+1})\sse S(x_n)\quad \mbx{and}\quad \vphi(x_{n+1})< \vphi(x_n)\;\mbx{ for all }\; n\in\Nat_0;\\
    {\rm (ii)}\quad &\mbx{there exist the limits}\;\alpha:= \lim_{n\to\infty}\vphi(x_n)=\lim_{n\to\infty}J(x_n)\in\Real;\\
   {\rm (iii)}\quad &x_{n+k}\in S(x_n)\quad\mbx{for all }\; n,k\in\Nat_0\,;\\
 {\rm (iv)}\quad &  (x_n)_{n\in\Nat_0} \mbx{ is right }\; K\mbx{-Cauchy}\,.
   \end{aligned}\eequ
If the space $ X $ is sequentially right $K$-complete and the function $\vphi$ is nearly lsc,  then the sequence  $(x_n)_{n\in\Nat_0}$
 is convergent to a point $ z\in X $     such that
 \bequ  \label{eq3.Pic}\begin{aligned}
   {\rm (i)}\quad& z\in S(x_n)\quad\mbx{and}\q S(z)\sse S(x_n)\quad\mbx{for all }\; n\in\Nat_0;\\
   {\rm (ii)}\quad& \vphi(y)=\vphi(z)=J(z)=\alpha\;\mbx{ and} \\
   {\rm (iii)}\quad& S(y)\sse\ov{\{y\}} \;   \mbx{ for all }\; y\in S(z)\,,
   \end{aligned}\eequ
  where $\alpha$ is given by \eqref{eq2.Pic}.{\rm(ii)}.
\een\end{prop}\begin{proof} Suppose that we have found $x_0,x_1,\dots,x_{m}$  satisfying the conditions (i) and (ii) from \eqref{eq01.Pic}.
If $\vphi(x_m)=J(x_m)$ then   $x_0,x_1,\dots,x_{m}$ satisfy   \eqref{eq01.Pic}.

If $\vphi(x_m)>J(x_m)$, then there exists $x_{m+1}\in S(x_m)$ such that $\vphi(x_{m+1})<(\vphi(x_n)+J(x_n))/2\,.$ Supposing that this procedure continues indefinitely, we find a sequence $(x_n)_{n\in\Nat_0}$ satisfying \eqref{eq1.Pic}.

1. Suppose that $x_0,x_1,\dots,x_{m}$ satisfy   \eqref{eq01.Pic} and let $z=x_m\,.$ Then, by \eqref{eq2.Sx}.(ii), $x_{k+1}\in S(x_k)$ implies $S(x_{k+1})\sse S(x_k)$ for $k=0,1,\dots,m-1.$

If $y\in S(z),$ then, by \eqref{eq2.Sx}.(ii),
$$
J(z)\le\vphi(y)\le\vphi(z)=J(z)\,.$$

It follows  $\vphi(y)=\vphi(z)=J(z)$ for all $y\in S(z).$

If $y\in S(z)$ and $x\in S(y)\sse S(z)$, then
$$
d(x,y)\le \vphi(y)-\vphi(x)=0\,,$$
so that $x\in \ov{\{y\}}$ and $S(y)\sse\ov{\{y\}}.$

We have shown that $z$ satisfies \eqref{eq02.Pic}.
\smallskip

  2. Suppose now that the sequence $(x_n)_{n\in\Nat_0} $ satisfies \eqref{eq1.Pic}.

By \eqref{eq2.Sx}.(ii), the relation $x_{n+1}\in S(x_n)$ implies
\bequs
S(x_{n+1})\sse S(x_n)\quad\mbx{for all } \; n\in\Nat_0.\eequs

Also, by \eqref{eq1.Pic},
$$
\vphi(x_{n+1})<\vphi(x_n)\quad\mbx{for all } \; n\in\Nat_0,$$
so \eqref{eq2.Pic}.(i) holds.

Since $\vphi$ is bounded below and, by (i),  $(\vphi(x_n))_{n\in\Nat_0}$ is strictly decreasing, there exists the limit
$$\alpha:=\lim_{n\to\infty}\vphi(x_n)=\inf_{n\in\Nat_0}\vphi(x_n)\in\Real\,.$$

By \eqref{eq1.Pic},
$$
2\vphi(x_{n+1})-\vphi(x_{n})<J(x_n)<\vphi(x_{n})\,,$$
for all $n\in\Nat_0.$ Letting  $n\to\infty$, one obtains
\bequ\label{eq4.Ek1-qm}
\lim_{n\to\infty}J(x_n)=\lim_{n\to\infty}\vphi(x_n)=\alpha\,.
\eequ

By  (i),
$$x_{n+k}\in S(x_{n+k})\sse S(x_n)\,,$$
proving (iii).

By (iii),
\bequ\label{eq5.Pic}
d(x_{n+k},x_{n})\le\vphi(x_{n})-\vphi(x_{n+k})\quad\mbx{for all }\; n,k\in\Nat_0\,.\eequ
 Since the sequence  $(\vphi(x_n))_{n\in\Nat_0}$  is Cauchy, this implies that $(x_n)_{n\in\Nat_0}$ is right $K$-Cauchy.

Suppose now that $X$ is sequentially  right $K$-complete and $\vphi$ is nearly lsc.

Taking into account the fact  that the function  $d(\cdot,x_n)$ is   lsc (Proposition \ref{p.top-qsm1}.3) and the sequence $(x_n)_{n\in\Nat_0}$ has pairwise distinct terms, the inequalities \eqref{eq5.Pic} yield
$$
d(z,x_n)+\vphi(z)\le\liminf_{k\to\infty}[\vphi(x_{n+k})+d(x_{n+k},x_n)]\le\vphi(x_n)\,,$$
which shows that
$
z\in S(x_n)$ and so $S(z)\sse S(x_n)$,  for all $n\in\Nat_0$.

If $y\in S(z)\sse S(x_n),$ then
$$
J(x_n)\le J(z)\le \vphi(y)\le\vphi(y)+d(y,x_n)\le\vphi(x_n)\,,$$
that is
$$
J(x_n)\le J(z)\le \vphi(y)\le \vphi(x_n)\,,$$
for all $n\in\Nat_0\,.$

Letting $n\to \infty$ and taking into account \eqref{eq2.Pic}.(ii), one obtains
$$
\vphi(y)=J(z)\,,$$
for all $y\in S(z).$

The proof of the inclusion \eqref{eq3.Pic}.(iii) is similar to that of \eqref{eq02.Pic}.(iii).
\end{proof}

\begin{remark}\label{re2.Pic} If the function $\vphi$ satisfies
$$
\vphi(x)>\inf\vphi(S(x)) \;\mbx{ for all }\; x\in \dom\vphi\,,$$
then, for every $x_0\in\dom\vphi$, there exists a Picard sequence $(x_n)_{n\in\Nat_0}$ satisfying \eqref{eq1.Pic} (and so \eqref{eq2.Pic} as well).

\end{remark}

 \section{ Ekeland, Takahashi and Caristi  principles in quasi-pseudometric spaces}\label{S.T0-qm}

 Along this  section we shall use the notation: for  a quasi-pseudometric space $X$,  a function $\vphi:X\to\Real\cup\{\infty\}$ and $x\in X$ put
\bequ\label{def.Sx-Jx}
S(x)=\{y\in X : \vphi(y)+d(y,x)\le \vphi(x)\}\quad\mbx{and}\q J(x)=\inf \vphi(S(x))\,.\eequ
Ekeland,  Takahashi and Caristi principles (see Theorem \ref{t1.Ek-Taka-Car})  can be expressed in terms of the sets $S(x)$ in the following form.
\begin{theo}\label{t2.Ek-Taka-Car} Let $(X,d)$ be a  complete  metric space and $\vphi:X\to\Real\cup\{\infty\}$ a proper bounded below  lsc function.
Then the following hold.
\ben
\item[\rm (wEk)] There exists $z\in X$ such that $S(z)=\{z\}.$
\item[\rm (Tak)] If $S(x)\sms\{x\}\ne \ety$ whenever $\vphi(x)>\inf\vphi(X),$ then $\vphi$ attains its minimum on $X$, i.e., there  exists
$z\in X$ such that $\vphi(z)=\inf\vphi(X).$
\item[\rm (Car)] If the mapping  $T:X\to X$ satisfies $Tx\in S(x) $ for all $x\in X,$ then $T$ has a fixed point in $X$, i.e.,   there  exists
$z\in X$ such that $Tz=z.$
\een\end{theo}

In the following we shall prove some quasi-pseudometric versions of these results.

\subsection{Ekeland variational principle}
We start by a version of weak Ekeland principle.
\begin{theo}\label{t.Ek1-qm} Let $(X,d)$ be a  sequentially  right $K$-complete quasi-pseudometric space and $\vphi:X\to\Real\cup\{\infty\}$ a proper bounded below nearly lsc function. Then there exists $z\in X$ such that
\bequ\label{eq1.Ek1-qm}\begin{aligned}
    \vphi(y)=\vphi(z) \;\mbx{ for all  }\; y\in S(z)\,.
\end{aligned}\eequ

In this case it follows that, for every  $y\in S(z)$,
\bequ\label{eq1b.Ek1-qm}\begin{aligned}
  {\rm (i)}\quad& S(y)\sse\ov{\{y\}}\quad\mbx{and}\\
   {\rm (ii)}\quad& \vphi(y)<\vphi(x)+d(x,y)\quad\mbx{for all }\; x\in X\sms S(y)\,.
\end{aligned}\eequ

\end{theo}\begin{proof}
By Proposition \ref{p1.Pic},  \eqref{eq02.Pic}.(iii) and   \eqref{eq3.Pic}.(ii), there exists $z\in X$ satisfying \eqref{eq1.Ek1-qm}.

Let us  prove  now  that \eqref{eq1.Ek1-qm} implies  \eqref{eq1b.Ek1-qm}.

If $y\in S(z)$ and $x\in S(y)\sse S(z)$, then, by \eqref{eq1.Ek1-qm},
$$
d(x,y)\le\vphi(y)-\vphi(x)=0\,,$$
so that, $x\in \ov{\{y\}},$ that is,  $S(y)\sse \ov{\{y\}}.$

If $x\in X\sms S(y)$, then, by the definition of the set $S(y)$,
$$
\vphi(y)<\vphi(x)+d(x,y)\,.$$
\end{proof}
\begin{remark} In \cite{karap-romag15} the following form of  the weak form of  EkVP is proved: Under the hypotheses of
Theorem \ref{t.Ek1-qm} there exists $z\in X$ such that
\bequ\label{wEk-KR}\begin{aligned}
  {\rm (i)}\quad& \vphi(z)\le \vphi(x)\;\mbx{ for   }\; x\in\ov{\{z\}}\;\mbx{ and}\\
  {\rm (ii)}\quad& \vphi(z)< \vphi(x)+d(x,z)\;\mbx{ for  }\; x\in X\sms \ov{\{z\}}\,.
\end{aligned}\eequ

Since  $S(z)\sse \ov{\{z\}}$, the relations  \eqref{eq1.Ek1-qm} and \eqref{eq1b.Ek1-qm} can be rewritten in the  form
\begin{align*}
{\rm (i')}\quad& \vphi(z)= \vphi(x)\;\mbx{ for   }\; x\in S(z);\\
{\rm (i'')}\quad& \vphi(z)< \vphi(x)\;\mbx{ for   }\; x\in\ov{\{z\}}\sms S(z)\;\mbx{ and}\\
  {\rm (ii)}\quad& \vphi(z)< \vphi(x)+d(x,z)\;\mbx{ for  }\; x\in X\sms \ov{\{z\}}\,,
  \end{align*}
  i.e., \eqref{wEk-KR}.(i) splits into (i$'$) and   (i$''$).

\end{remark}

Supposing that $\vphi$ is lsc  one can obtain  the full version of Ekeland variational principle.

\begin{theo}\label{t.Ek2-qm} Let $(X,d)$ be a sequentially right $K$-complete quasi-pseudometric space and $\vphi:X\to\Real\cup\{\infty\}$ a proper bounded below  lsc function.
Let $\epsic,\lbda >0$ and let $x_0\in X$ be such that
\bequ\label{eq1.Ek2-qm}
\vphi(x_0)\le\epsic+\inf\vphi(X)\,.
\eequ

Then there exists $z\in X$ such that
\bequ\label{eq2.Ek2-qm}\begin{aligned}
{\rm (i)}\quad&  \vphi(z)+\frac{\epsic}{\lbda}d(z,x_0)\le\vphi(x_0)\q(\mbx{and so }\; \vphi(z)\le\vphi(x_0));\\
{\rm (ii)}\quad& d(z,x_0)\le\lbda;\\
{\rm (iii)}\quad& \vphi(y)=\vphi(z) \;\mbx{ for all  }\; y\in S(z);\\
{\rm (iv)}\quad& \vphi(z)<\vphi(x)+\frac{\epsic}{\lbda}d(x,z)\quad\mbx{for all }\; x\in X\sms S(z)\,.
\end{aligned}\eequ\end{theo}\begin{proof}
For convenience, put $\gamma=\epsic/\lbda$ and $d_\gamma=\gamma d.$ Then $d_{\gamma}$ is a quasi-pseudometric on $X$ Lipschitz equivalent to $d$, so that $(X,d_{\gamma})$ is sequentially right $K$-complete too.

Let
$$
X_0=\{x\in X : \vphi(x)\le \vphi(x_0)+d_{\gamma}(x_0,x)\}\,.$$

\emph{Claim} I. \emph{The set} $X_0$ \emph{is closed and} $x_0\in X_0\,.$\smallskip

Indeed, let $(x_n)$ be a sequence in $X_0$, $d_{\gamma}$-convergent to some $x\in X$, i.e., $\lim_{n\to\infty}d_{\gamma}(x,x_n)=0.$ Then
\begin{align*}
  \vphi(x_n) &\le \vphi(x_0)+d_{\gamma}(x_0,x_n)\\
  &\le \vphi(x_0)+d_{\gamma}(x_0,x)+d_{\gamma}(x,x_n)\,,
\end{align*}
for all $n\in\Nat.$ Taking into account the lsc of the function $\vphi$, one obtains
$$
\vphi(x)\le\liminf_{n\to\infty}\vphi(x_n)\le \vphi(x_0)+d_{\gamma}(x_0,x)\,,$$
which shows that $x\in X_0.$

It is obvious that $x_0\in X_0\,.$

For $y\in X_0$ put
$$
 S_{X_0}(y):=\{x\in X_0 : \vphi(x)+d_{\gamma}(x,y)\le \vphi(y)\}=X_0\cap S(y)\,.$$

\emph{Claim} II. \emph{For every} $y\in X_0  $  \emph{ and} $x\in X\sms X_0,$
\bequ\label{eq3.Ek2-qm}\begin{aligned}
 {\rm (i)}\quad& \vphi(y) <\vphi(x)+d_{\gamma}(x,y)\;\mbx{ and}\\
  {\rm (ii)}\quad& S_{X_0}(y)=\{x\in X : \vphi(x)+d_{\gamma}(x,y)\le \vphi(y)\}=S(y)\,.
\end{aligned}\eequ

Indeed,
$$
x\in X\sms X_0\iff \vphi(x_0)+d_{\gamma}(x_0,x)<\vphi(x)\,,$$
so that, taking into account that $y\in X_0$, we obtain
\begin{align*}
  \vphi(y) &\le \vphi(x_0)+d_{\gamma}(x_0,y)\\
   &\le \vphi(x_0)+d_{\gamma}(x_0,x)+d_{\gamma}(x,y)\\
   &<\vphi(x)+d_{\gamma}(x,y)\,.
\end{align*}

The inequality \eqref{eq3.Ek2-qm}.(i) implies that
\begin{align*}
  S_{X_0}(y)&=\{x\in X_0 : \vphi(x)+d_{\gamma}(x,y)\le \vphi(y)\}\\
  &=\{x\in X : \vphi(x)+d_{\gamma}(x,y)\le \vphi(y)\} =S(y)\,.
\end{align*}

Applying  Proposition \ref{p1.Pic}  to $(X_0,d_{\gamma})$ we find an element $z\in X_0$ such that
\bequ\label{eq4.Ek2-qm}
\vphi(y)=\vphi(z)\;\mbx{ for all }\; y\in S_{X_0}(z)=S(z)\,.
\eequ

This shows that $z$ satisfies  \eqref{eq2.Ek2-qm}.(iii).

Now, by \eqref{eq02.Pic}.(ii)and \eqref{eq3.Pic}.(i), $z\in S_{X_0}(x_0)$ which is equivalent to \eqref{eq2.Ek2-qm}.(i).

Also, \eqref{eq2.Ek2-qm}.(i)  and \eqref{eq1.Ek2-qm} imply
$$
\frac{\epsic}{\lbda}d(z,x_0)\le \vphi(x_0)-\vphi(z)\le\epsic\,,$$
so that
$$
d(z,x_0)\le  \lbda\,,$$
i.e., \eqref{eq2.Ek2-qm}.(ii) holds too.

 The inequality  \eqref{eq3.Pic}.(iv) follows from  the definition of the set $S(z).$
 \end{proof}

\subsection{Takahashi principle}

 \begin{theo}\label{t1.Taka-qm}
Let $(X,d)$ be a sequentially  right $K$-complete  quasi-metric space and $\vphi:X\to\Real\cup\{\infty\}$ a proper, bounded below and nearly lsc function. Suppose that,  for every $x\in X$,
\bequ\label{eq1.Taka-qm}
\vphi(x)>\inf \vphi(X)   \Ra   \exists y\in S(x),\; \vphi(y)<\vphi(x)\,.\eequ

Then there exists $z\in X$ such that $\vphi(z)=\inf \vphi(X)$, i.e., the function $\vphi$ attains its minimum  on $X$.
\end{theo}\begin{proof} Suppose, by contradiction, that
\bequ\label{eq0.Taka-qm}
\vphi(x)>\inf \vphi(X)\\,
\eequ
for all $x\in X$. Then, by \eqref{eq1.Taka-qm},
\bequ\label{eq2.Taka-qm}
\forall x\in X,\;   \exists y\in S(x),\; \vphi(y)<\vphi(x)\,,\eequ
or, equivalently,
\bequ\label{eq3.Taka-qm}
\vphi(x)>\inf \vphi (S(x))\,,
\eequ
for every $x\in X$.

Let   $x_0\in\dom \vphi .$ By \eqref{eq3.Taka-qm}, Remark \ref{re2.Pic}  and Proposition \ref{p1.Pic}.2,  there exists a sequence $(x_n)_{n\in\Nat_0}$ satisfying
  \eqref{eq2.Pic} and \eqref{eq3.Pic}.
  If $z=\lim_{n\to\infty}x_n$, then,   by \eqref{eq2.Taka-qm},   there exists $y\in S(z)$ such that

  \bequ\label{eq4.Taka-quasi-pseudometric-n}
\vphi(y)< \vphi(x)\,.\eequ

By \eqref{eq1.Pic}.(ii),
$$
2\vphi(x_{n})- \vphi(x_{n-1})<\inf \vphi(S(x_{n-1}))\le \vphi(y)\,,$$
for all $n\in\Nat$  (because $y\in S(x_{n-1})$, by \eqref{eq3.Pic}.(i)).

Taking into account the nearly lsc of the function $\vphi$ it follows
$$
\vphi(x)\le \lim_{n\to\infty}\vphi(x_n)=\lim_{n\to\infty}\big(2\vphi(x_{n})- \vphi(x_{n-1})\big)\le \vphi(y)\,,$$
in contradiction to \eqref{eq4.Taka-quasi-pseudometric-n}.

Consequently, the hypothesis \eqref{eq0.Taka-qm} leads to a contradiction, so it must  exist a point $z\in X$ such that $\vphi(z)=\inf \vphi(X).$
 \end{proof}

 \begin{corol}\label{c1.Taka-qm}
Suppose that  $(X,d)$   and $\vphi:X\to\Real\cup\{\infty\}$ satisfy the hypotheses of Theorem \ref{t1.Taka-qm}.
If,  for every $x\in X$,
\bequ\label{eq1b.Taka-qm}
\vphi(x)>\inf \vphi(X)   \Ra   \exists y\in X\sms\ov{\{x\}}\;\mbx{ such that }\;  \vphi(y)+d(y,x)\le \vphi(x)\,,\eequ
then the function $\vphi$ attains its minimum on $X$.
\end{corol}\begin{proof} Condition \eqref{eq1b.Taka-qm} means that, for every $x\in X$,
\bequs
\vphi(x)>\inf \vphi(X)   \Ra  S(x)\sms\ov{\{x\}}\ne\ety\,.\eequs
 By \eqref{eq2.Sx}.(iii),
$$
y\in S(x)\sms\ov{\{x\}} \Ra  \vphi(y)<\vphi(x) \,,$$
so we can apply    Theorem \ref{t1.Taka-qm} to conclude.
\end{proof}

\subsection{Caristi fixed point theorem} We present both single-valued and set-valued versions of Caristi fixed point theorem
\begin{theo}[Caristi's theorem]\label{t1.Caristi-qm} Let $(X,d)$ be a sequentially right $K$-complete quasi-pseudometric space and $\vphi:X\to\Real\cup\{\infty\}$ a proper bounded below
  nearly lsc function.
  \ben\item[\rm 1.] If the mapping $T:X\to X$ satisfies
  \bequ\label{eq1.Car-qm}
  d(Tx,x)+\vphi(Tx)\le\vphi(x)\,,
  \eequ
  for all $x\in X,$ then there exists $z\in X$ such that $\vphi(Tz)=\vphi(z).$
  \item[\rm 2.]   If   $T:X\rightrightarrows X$ is a set-valued mapping such that
  \bequ\label{eq2.Car-qm}
  S(x)\cap Tx\ne \ety\,,
  \eequ
  for every $x\in X,$ then there exists $z\in X$ such that $\vphi(z)\in \vphi(Tz).$
\een\end{theo}\begin{proof} Observe that
 condition \eqref{eq1.Car-qm} is equivalent to
\bequ\label{eq3.Car-qm}
 Tx\in S(x) \,,
  \eequ
  for all $x\in X.$ This shows that \eqref{eq2.Car-qm} is an extension of  \eqref{eq1.Car-qm}, so it suffices to prove 2.

  By Proposition \ref{p1.Pic}, \eqref{eq02.Pic}(iii) and \eqref{eq3.Pic}(ii), there exists $z\in X$ such that
  $\vphi(y)=\vphi(z)$ for all $y\in S(z)$. By \eqref{eq2.Car-qm}, there exists $y\in S(z)\cap Tz.$ But then
    $$
    \vphi(z)=\vphi(y)\in\vphi(Tz)\,.$$
\end{proof}

\subsection{The equivalence of principles and completeness}
We prove   the equivalence between Ekeland, Takahashi and Caristi principles.
\begin{theo} \label{t1.wEk-Taka-Car}
Let  $(X,d)$ be  a quasi-pseudometric space
 and $\vphi:X\to\Real\cup\{\infty\}$  a proper bounded below   function.
 Then the following statements are equivalent.
 \ben
 \item[\rm (wEk)]The following holds
 \bequ\label{eq1.wEk-Taka}
   \exists z\in X,\; \forall   y\in S(z),\;\vphi(y) = \vphi(z) \,.
    \eequ
 \item[\rm (Tak)] \quad The following holds
 \bequ\label{eq2.wEk-Taka}\begin{aligned}
 \big\{\forall x\in X,\; &\big[\inf\vphi(X)<\vphi(x)\,\Ra\, \exists y\in S(x),\; \vphi(y)<\vphi(x)\big]\big\}\\&\Lra\;\exists z\in X,\; \vphi(z)=\inf\vphi(X)\,.\end{aligned}\eequ
  \item[\rm (Car)] \quad If the mapping $T:X\to X$ satisfies
  \bequ\label{eq1.wEk-Car}
  Tx\in S(x)\;\mbx{ for all }\; x\in X,
  \eequ
  then there exists $z\in X$ such that $\vphi(Tz)=\vphi(z).$
 \een\end{theo}\begin{proof} (wEK)$\iff$ (Tak).
 
 The proof is based on the following rules from mathematical logic:
\bequs
(p\rar q)\lrar (\neg p\vee q)\,,
\eequs
so that
\bequ\label{eq1.neg}
\neg(p\rar q)\lrar (p\wedge\neg q)\,.
\eequ

Observe that
\bequ\label{eq3.wEk-Taka}
\big(\forall x\in X,\; \exists y\in S(x),\; \vphi(y)<\vphi(x)\big) \Ra  \big(\forall x\in X,\; \vphi(x)>\inf\vphi(X)\big)\,.\eequ

Indeed, for every $x\in X$ take $y\in S(x)$ such that $\vphi(y)<\vphi(x).$ Then
$$
\inf\vphi(X)\le \vphi(y)<\vphi(x)\,.$$

For convenience, denote by Ta1 the expression
$$
 \forall x\in X,\; \big[\inf\vphi(X)<\vphi(x)\,\Ra\, \exists y\in S(x),\; \vphi(y)<\vphi(x)\big]\,.$$

Based on \eqref{eq1.neg} and   \eqref{eq3.wEk-Taka}, one obtains:
\begin{align*}
  \neg\mbox{(Tak)}&\iff \mbx{(Ta1)}\,\wedge \big(\forall z\in X,\; \vphi(z)>\inf\vphi(X)\big)\\
  &\iff  \big(\forall x\in X,\; \exists y\in S(x),\; \vphi(y)<\vphi(x)\big)\wedge \big(\forall z\in X,\; \vphi(z)>\inf\vphi(X)\big)\\
    &\iff \forall x\in X,\; \exists y\in S(x),\; \vphi(y)<\vphi(x)\,.
   \end{align*}

  Since,   by \eqref{eq2.Sx}.(ii), $\vphi(y)\le\vphi(x) $ for every $y\in S(x)$, it follows that
   \beqs
   \big((y\in S(x))\wedge (\vphi(y)\ne\vphi(x)\big) \iff \big((y\in S(x))\wedge (\vphi(y)<\vphi(x)\big)\,.
   \eeqs

   But then
  \bequ\label{non-wEk} \begin{aligned}
     \neg{\rm (wEk)}&\iff  \forall x\in X,\; \exists y\in S(x),\;\vphi(y)\ne\vphi(x)\\
     &\iff  \forall x\in X,\; \exists y\in S(x),\;\vphi(y)<\vphi(x)\\
       &\iff \neg\mbox{(Tak)}
   \end{aligned}\eequ
   
   (wEk)\;$\Ra$\; (Car). 
   
   Suppose that $T:X\to X$ satisfies \eqref{eq1.wEk-Car}. By (wEk) there exists $z\in X$ such that
 $\vphi(x)=\vphi(z)$ for all $x\in S(z).$ Since, by hypothesis,  $Tz\in S(z)$, it follows $\vphi(Tz)=\vphi(z).$\smallskip

$\neg$(wEk) \;$\Ra$\; $\neg$(Car). 

By \eqref{non-wEk},  $\neg$(wEk) is equivalent to
\bequ\label{eq.non-wEk}
 \forall x\in X,\; \exists y_x\in S(x),\;\vphi(y_x)<\vphi(x)\,.\eequ

 Define $T:X\to X$ by $Tx=y_x$, where $y_x$ is given by \eqref{eq.non-wEk},\,$x\in X.$ Then $Tx\in S(x)$ for every $x\in X$ but $\vphi(Tx)<\vphi(x)$
 for all $x\in X$, i.e., the assertion  (Car) fails.
 
\end{proof}

Finally we show that the validity of each of these principles is further equivalent to the sequential right $K$-completeness of the quasi-pseudometric space $X$.
\begin{theo} For a quasi-pseudometric space  $(X,d)$  the following are equivalent.
\ben\item[\rm 1.] The space $(X,d)$ is sequentially right $K$-complete.
\item[\rm 2.]{\rm (Ekeland variational principle - weak form)}  For every proper bounded below nearly lsc function $\vphi:X\to\Real\cup\{\infty\}$ there exists $z\in X$ such that
\bequ\label{eq1.wEk-compl}
\vphi(x)=\vphi(z)\;\mbx{ for all }\; x\in S(z)\,.
\eequ
\item[\rm 3.]{\rm (Takahashi minimization principle)}   Every proper bounded below nearly lsc function $\vphi:X\to\Real\cup\{\infty\}$ such that, for every $x\in X$,
\bequs
\vphi(x)>\inf \vphi(X)   \Ra   \exists y\in S(x),\; \vphi(y)<\vphi(x)\,,\eequs
 attains its minimum on $X$.
\item[\rm 4.]{\rm (Caristi fixed point theorem)} For every proper bounded below nearly lsc function $\vphi:X\to\Real\cup\{\infty\}$ and every mapping $T:X\to X$ such that
\bequs
Tx\in S(x)\;\mbx{ for all }\; x\in X,
\eequs
 there exists $z\in X$ such that $\vphi(Tz)=\vphi(z).$
\een\end{theo}\begin{proof}
  The equivalences $2\iff 3\iff 4$ are contained in Theorem \ref{t1.wEk-Taka-Car} (even in a stronger form - with the same function $\vphi$.)
  
  The implication $1\;\Ra 2$  is contained in Theorem \ref{t.Ek1-qm}.\smallskip
  
  2\;$\Ra\;$1. 
  
 The proof is inspired from \cite{karap-romag15}. We proceed by contradiction. Suppose that there exists a right $K$-Cauchy sequence $(x_n)_{n\in\Nat}$ in $X$ which does not converge. Then $(x_n)$ has no cluster points (by Proposition \ref{p1.rKC}),
so it  does not contain constant subsequences.
  Passing to a subsequence if necessary, we can suppose    that  further
 \bequ\label{eq.compl-seq}
 d(x_{n+1},x_n)< \frac1{2^{n+1}},\quad\mbx{ for all } n\in\Nat.
 \eequ

The set
 $$
 B:=\{x_n : n\in\Nat\}\,$$
 is closed. Indeed, if there exists $x\in\ov B\sms B,$ then $x$ will be a cluster point for $(x_n)$ in contradiction to the hypothesis.

 Define $\vphi:X\to \Real$ by
 $$
\vphi(x)=\begin{cases}\frac1{2^{n-1}}\quad&\mbox{if }\;  x=x_n \;\mbox{ for  some }\; n\in\Nat,  \\
\infty \qquad\quad&\mbox{for}\quad x\in X\setminus B.
\end{cases}$$

The function $\vphi$ is nearly lsc. Indeed,  let $(y_n)$ be a sequence with distinct terms  converging to a point $x\in X.$

If $x\in X\setminus B,$  then, since $B$ is closed, the sequence $(y_n)$   must be eventually in $X\setminus B,$ and so $\lim_n\vphi(y_n)=\infty=\vphi(x).$

Suppose now that  $x=x_k$ for some $k\in\Nat.$

If the set $\{n\in\Nat : y_n\in B\}$
is infinite, then $x_k$ will be a cluster point of the sequence $(x_n)$. Consequently, only finitely many terms $y_n$ belong to $B$. This implies that there exists $n_0 \in\Nat$ such that $y_n\in X\sms B$ for all $n\ge n_0, $ so that  $\vphi(x_k)<\infty=\lim_n \vphi(y_n)$.

We have $\dom\vphi=B.$ If $z=x_k\in B$, for some $k\in\Nat$, then, by \eqref{eq.compl-seq} and the definition of the function $\vphi$,
$$
\vphi(x_{k+1})+d(x_{k+1},x_k)<\frac1{2^{k}}+\frac1{2^{k+1}}= \frac{3}{2^{k+1}}<\frac1{2^{k-1}}=\vphi(x_k),$$
which shows that $x_{k+1}\in S(x_k)$. Since $\vphi(x_k)=2^{-k+1} > 2^{-k}=\vphi(x_{k+1})$, it follows that \eqref{eq1.wEk-compl} fails for every $z\in\dom\vphi.$
\end{proof}

\subsection{The case of $T_1$ quasi-metric spaces}
Let us notice  that a topological space $(X,\tau)$ is $T_1$ if and only if $\ov{\{x\}}=\{x\}$ for all $x\in  X.$  A quasi-pseudometric space $(X,d)$ is $T_1$  (i.e., the topology $\tau_d$ is $T_1$) if and only if $d(x,y)>0$ for all $x\ne y$ in $X$. It follows that a $T_1$ quasi-pseudometric space is a quasi-metric space. By Proposition \ref{p1.lsc-n}.3 a nearly lsc function on a $T_1$ quasi-metric space is lsc. \par
Taking into account these remarks, the results proved for arbitrary quasi-pseudometric spaces take the following form in the $T_1$ case.

\begin{theo}\label{t1.Taka-Ek-Car}
Let $(X,d)$ be a  sequentially right $K$-complete  $T_1$ quasi-metric space and $\vphi:X\to\Real\cup\{\infty\}$ a proper bounded below  lsc function. The following are true.
\ben
\item[\rm 1.] {\rm (Ekeland variational principle -weak form, \cite{cobz11})}
 There exists $z\in X$ such that
 \bequ\label{eq1.wEk-T1}
   \vphi(z) < \vphi(x)+  d(x,z)\quad \mbx{ for all }\;   x\in X\setminus \{z\}\,.
    \eequ
\item[\rm 2.] {\rm (Takahashi principle, \cite{kassay19})}
If for every $x\in X$,
\bequ\label{eq1.Taka-T1}
\vphi(x)>\inf \vphi(X)   \Ra   \exists y\in X\sms\{x\}\;\mbx{ such that }\;  \vphi(y)+d(y,x)\le \vphi(x)\,,\eequ
then there exists $z\in X$ such that $\vphi(z)=\inf \vphi(X).$
\item[\rm 3.] {\rm (Caristi fixed point theorem)}\ben
\item[\rm(a)] If the mapping $T:X\to X$ satisfies
  \bequ\label{eq1.Car-T1}
  d(Tx,x)+\vphi(Tx)\le\vphi(x)\,,
  \eequ
  for all $x\in X,$ then there exists $z\in X$ such that $Tz= z.$
 \item[\rm (b)] If   $T:X\rightrightarrows X$ is a set-valued mapping such that
  \bequ\label{eq2.Car-T1}
  S(x)\cap Tx\ne \ety\,,
  \eequ
  for every $x\in X,$ then there exists $z\in X$ such that $z\in Tz.$
\een
\een \end{theo}\begin{proof}
  1. By Theorem \ref{t.Ek1-qm} there exists $z\in X$ satisfying \eqref{eq1b.Ek1-qm}. Since $X$ is $T_1$, $\ov{\{z\}}=\{z\}$, so that $S(z)=\{z\}.$
  Taking into account this equality, \eqref{eq1.wEk-T1} is equivalent to \eqref{eq1b.Ek1-qm}.(iii) for $y=z.$

  2. Condition \eqref{eq1.Taka-T1} says that $S(x)\sms\{x\}\ne \ety$ whenever $\vphi(x)>\inf\vphi(X).$  Since $X$ is $T_1$, $d(y,x)>0$, so that
   $$\vphi(y)<\vphi(y)+d(y,x)\le \vphi(x)\,,$$
   for every $y\in S(x)\sms\{x\}.$ This shows that condition \eqref{eq1.Taka-qm} is verified by $\vphi.$

   3.  As we have noticed, condition \eqref{eq1.Car-T1} means that $Tx\in S(x)$ for every $x\in X,$ so it suffices to give the proof only for    set-valued  $T$. The proof is the same as that of Theorem \ref{t1.Caristi-qm}, taking into account that $S(z)=\{z\}$.

   Indeed,  $S(z)=\{z\}$ and $ S(z)\cap Tz\ne\ety$ imply $z\in Tz.$
\end{proof}

 \end{document}